\documentclass[11pt,reqno]{amsart}
\usepackage{amsfonts}
\usepackage{amsmath}
\usepackage{amssymb}
\usepackage{color}
\usepackage{graphicx}
\usepackage{xspace}
\usepackage{mathrsfs}
\usepackage{hyperref}
\usepackage[cal=boondoxo]{mathalfa}
\usepackage{forest}

\newfont{\indic}{bbmss12}
\def\ind#1{\hbox{{\indic 1}$_{#1}$}}

\usepackage{xspace}

\newcommand{\id}{\mathrm{id}}
\newcommand{\R}{\mathbb R}
\newcommand{\N}{\mathbb N}
\newcommand{\E}{\mathbb E}
\newcommand{\CA}{\mathcal A}
\newcommand{\CC}{\mathcal C}
\newcommand{\CX}{\mathcal X}

\newcommand{\CP}{\mathcal P}
\newcommand{\T}{\mathcal T}
\newcommand{\F}{\mathcal F}

\newcommand{\M}{{\mathcal M}}

\def\MS{\mathscr{M}}
\def\Wick#1{\,\colon\! #1 \colon}

\def\Cum{\mathbf{E}_c}
\DeclareMathOperator{\ev}{ev}

\usetikzlibrary{external}
\forestset{
  decor/.style = {
    edge label/.expanded = { node[midway, anchor=center, fill=white, inner sep=0.25ex, minimum size=0ex, font=\unexpanded{\tiny}]{$#1$} }
  },
  root/.style = {minimum size = 0.55ex},
  decorated/.style = {
    for tree = {
      circle, fill, inner sep = 0ex, minimum size = 0.55ex,
      grow' = north, l = 0, l sep = 2.5ex, s sep=1.5em,
      fit=tight,
      delay = {decor/.option=content, content=}
    }
  },
  default preamble = {decorated, root},
  begin draw/.code={\begin{tikzpicture}[baseline={([yshift=-0.5ex]current bounding box.center)}]}
}
\newcommand{\tRoot}{\Forest{[]}}
\newcommand{\tI}[1]{\Forest{[ [#1]]}}
\newcommand{\tII}[2]{\Forest{[ [#1[#2]] ]}}
\newcommand{\tV}[2]{\Forest{[ [#1][#2] ]}}

\newcommand{\tIV}[3]{\Forest{[ [#1[#2][#3]] ]}}

\newtheorem{theorem}{Theorem}[section]

\newtheorem{lemma}[theorem]{Lemma}
\newtheorem{proposition}[theorem]{Proposition}

\newtheorem{remark}[theorem]{Remark}

\newtheorem{definition}[theorem]{Definition}
\newtheorem{example}[theorem]{Example}

\title[Hopf-algebraic deformations of products]{Hopf-algebraic deformations of products\\ and Wick polynomials}

\author[K.~Ebrahimi-Fard]{K.~Ebrahimi-Fard}
\address{Dept.~of Mathematical Sciences, 
		Norwegian University of Science and Technology (NTNU),
		7491 Trondheim, Norway. 
		{\tiny{On leave from UHA, Mulhouse, France.}}}
         \email{kurusch.ebrahimi-fard@ntnu.no}         
         \urladdr{https://folk.ntnu.no/kurusche/}

\author[F.~Patras]{F.~Patras}
\address{Univ.~C\^ote d'Azur, CNRS,
         		UMR 7351,
         		Parc Valrose,
         		06108 Nice Cedex 02, France.}
\email{patras@unice.fr}
\urladdr{www-math.unice.fr/$\sim$patras}

\author[N.~Tapia]{N.~Tapia}
\address{Departamento de Ingenier\'ia Matem\'atica, 
		Universidad de Chile, 
		Santiago, Chile.}
\email{ntapia@dim.uchile.cl}
\urladdr{http://www.dim.uchile.cl/$\sim$ntapia}

\author[L.~Zambotti]{L.~Zambotti}
\address{Laboratoire de Probabilit\'es, Statistique et Mod\'elisation (UMR 8001), Sorbonne Universit\'e, 
		Paris, France}       
         \email{lorenzo.zambotti@upmc.fr}
         \urladdr{http://www.lpsm.paris/pageperso/zambotti}

\date{\today}

\begin{document}

\begin{abstract}
We present an approach to cumulant--moment relations 
and Wick polynomials based on extensive use of convolution products of linear functionals 
on a coalgebra. This allows, in particular, to understand the construction of Wick polynomials as the 
result of a Hopf algebra deformation under the action of linear automorphisms induced by 
multivariate moments associated to an arbitrary family of random variables with moments of all 
orders. We also generalize the notion of deformed product in order to discuss how these ideas
appear in the recent theory of regularity structures.
\end{abstract}

\maketitle






\section{Introduction}
\label{sec:intro}

Chaos expansions and Wick products have notoriously been thought of as key steps in the renormalization process in perturbative quantum field theory (QFT). The technical reason for this is that they allow to remove contributions to amplitudes (say, probability transitions between two physical states) that come from so-called diagonal terms -- from which divergences in the calculation of those amplitudes may originate. Rota and Wallstrom \cite{RoWa} addressed these issues from a strictly combinatorial point of view using, in particular, the structure of the lattice of set partitions. These are the same techniques that are currently used intensively in the approach by Peccati and Taqqu in the context of Wiener chaos and related phenomena. We refer to their book \cite{Pec-Taq} for a detailed study and the classical results on the subject, as well as for a comprehensive bibliography and historical survey.

Recently, the interest in the fine structure of cumulants and Wick products for non-Gaussian variables has been revived, since they both play important roles in M.~Hairer's theory of regularity structures \cite{reg}. See, for instance, references \cite{Hai-Cha,Hai-She}. The progress in these works relies essentially on describing the underlying algebraic structures in a transparent way. Indeed, the combinatorial complexity of the corresponding renormalization process requires the introduction of group-theoretical methods such as, for instance, renormalization group actions and comodule Hopf algebra structures \cite{bhz}. Another reference of interest on generalized Wick polynomials in view of the forthcoming developments is the recent paper \cite{Luk-Marc}. 

Starting from these remarks, in this article we shall discuss algebraic constructions related to moment--cumulant relations as well as Wick products, using Hopf algebra techniques. A key observation, that seems to be new in spite of being elementary and powerful, relates to the interpretation of multivariate moments of a family of random variables as a linear form on a suitable Hopf algebra. It turns out that the operation of convolution with this linear form happens to encode much of the theory of Wick products and polynomials. This approach enlightens the classical theory, as various structure theorems in the theory of chaos expansions follow immediately from elementary Hopf algebraic constructions, and therefore are given by the latter a group-theoretical meaning. Our methods should be compared with the combinatorial approach in \cite{Pec-Taq}. 

Our approach has been partially motivated by similarities with methods that have been developed for bosonic and fermionic Fock spaces by C.~Brouder et al.~\cite{BruFra,brouderpat} to deal with interacting fields and non-trivial vacua in perturbative QFT. This is not surprising since, whereas the combinatorics of Gaussian families is reflected in the computation of averages of creation and annihilation operators over the vacuum in QFT, combinatorial properties of non-Gaussian families correspond instead to averages over non-trivial vacua.

The main idea of this paper is that the coproduct of a bialgebra allows to deform the product and that this permits to encode interesting constructions such as generalized Wick polynomials. In the last sections of this paper, we show how the above ideas can be used in more general contexts, which include regularity structures. Regarding the latter, we mention that these ideas have been used and greatly expanded in a series of recent papers \cite{reg,bhz,Hai-Cha} on renormalization of regularity structures. These papers handle products of {\it random} distributions which can be ill-defined and need to be {\it renormalized}. The procedure is rather delicate since the renormalization, which we rather call {\it deformation} in this paper, must preserve other algebraic and analytical structures. Without explaining in detail the rather complex constructions appearing in \cite{reg,bhz,Hai-Cha}, we describe how one can formalize this deformed (renormalized) product of distributions by means of a comodule structure.


\subsection{Generalized Wick polynomials}
\label{ssect:mainresults}

The main results of the first part of this paper (Theorems \ref{deformation} and \ref{ww}) are multivariate generalizations of the following statements for single real-valued random variable $X$ with finite moments of all orders.

We denote by $H:=\R[x]$ the algebra of polynomials in the variable $x$, endowed with the standard product
\begin{equation}\label{xn}
	x^n\cdot x^m:=x^{n+m}, 
\end{equation}
for $n,m\geq 0$. We equip $H$ with the cocommutative coproduct $\Delta: H\to H\otimes H$ defined by
\begin{equation}\label{deltan}
	\Delta x^n := \sum_{k=0}^n \binom{n}{k}x^{n-k}\otimes x^k.
\end{equation}
Product \eqref{xn} and coproduct \eqref{deltan} together define a connected graded commutative and cocommutative bialgebra, and therefore a Hopf algebra structure on $H$. On the dual space $H^*$ a dual product $\alpha \star \beta \in H^*$ can be defined in terms of \eqref{deltan}
\begin{equation}\label{convolproduct}
	(\alpha\star\beta)(x^n):=(\alpha\otimes\beta)\Delta x^n,
\end{equation}
for $\alpha,\beta\in H^*$. This product is commutative and associative, and the space ${\mathcal G}(H):=\{\lambda\in H^*: \lambda(1)=1\}$ forms a group for this multiplication law.

We define the functional $\mu\in H^*$ given by $\mu(x^n):=\mu_n=\E(X^n)$. Then $\mu\in
{\mathcal G}(H)$ and therefore its inverse $\mu^{-1}$ in ${\mathcal G}(H)$ is well 
defined. 

\begin{theorem}[Wick polynomials]
We define $W:=\mu^{-1}\star\id:H\to H$, i.e., the linear operator such that
\begin{equation}
\label{W(xn)}
	W(x^n)=(\mu^{-1}\otimes\id)\Delta x^n
	=\sum_{k=0}^n \binom{n}{k}\mu^{-1}(x^{n-k})\, x^k.
\end{equation}
Then 
\begin{itemize}
\item $W:H\to H$ is the only linear operator such that
\begin{equation}
\label{hermite}
	W(1)=1, 
	\qquad 
	\frac{\rm d}{{\rm d} x} \circ W = W \circ \frac{\rm d}{{\rm d} x}, 
	\qquad 
	\mu(W(x^n))=0,
\end{equation}
for all $n\geq 1$.
\item $W:H\to H$ is the only linear operator such that for all $n\geq 0$
\[
	x^n = (\mu\otimes W)\Delta x^n 
	= \sum_{k=0}^n \binom{n}{k}\mu(x^{n-k})\, W(x^k).
\] 
\end{itemize}
\end{theorem}

We call $W(x^n) \in H$ the Wick polynomial of degree $n$ associated to the law of $X$. If $X$ is a standard Gaussian random variable then the recurrence \eqref{hermite} shows that $W(x^n)$ is the Hermite polynomial $H_n$. Therefore \eqref{W(xn)} gives an explicit formula for such generalized Wick polynomials in terms of the inverse $\mu^{-1}$ of the linear functional $\mu$ in the group ${\mathcal G}(H)$. 

The Wick polynomial $W$ permits to define a {\it deformation} of the Hopf algebra $H$. 

\begin{theorem}\label{1.2}
The linear operator $W:H\to H$ has a composition inverse $W^{-1}:H\to H$ given by $W^{-1}=\mu\star\id$. If we define for $n,m\geq 0$ the product
\[
	x^n\cdot_\mu x^m := W ( W^{-1}(x^n)\cdot W^{-1}(x^m) ),
\]
and define similarly a twisted coproduct $\Delta_\mu$, then $H$ endowed with $\cdot_\mu$,\ $\Delta_\mu$ and $\varepsilon_\mu:=\mu$ is a bicommutative Hopf algebra. The map $W$ becomes an isomorphism of Hopf algebras. In particular
\[
	W(x^{n_1+\cdots+n_k})
	=W(x^{n_1})\cdot_\mu W(x^{n_2})\cdot_\mu\cdots\ \cdot_\mu W(x^{n_k}),
\]
for all $n_1,\ldots,n_k\in\N$.
\end{theorem}


\smallskip 

We recall that in the case of a single random variable $X$ with finite moments of all orders, the sequence $(\kappa_n)_{n\geq 0}$ of cumulants of $X$ is defined by the following formal power series relation between exponential generating functions
\begin{equation}
\label{expo}
	\exp\left(\sum_{n\geq 0} \frac{t^n}{n!}\,\kappa_n\right) 
	= \sum_{n\geq 0} \frac{t^n}{n!}\mu_n,
\end{equation}
where $t$ is a formal variable and $\mu_n=\E(X^n)$ is the $n$th-order moment of $X$. Note that $\mu_0=1$ and $\kappa_0=0$. Equation \eqref{expo} is equivalent to the classical recursion 
\begin{equation}
\label{Bell}
	\mu_n = \sum\limits_{m=1}^{n}\binom{n-1}{m-1} \kappa_{m} \mu_{n-m}.
\end{equation}
In fact, equation \eqref{expo} together with \eqref{Bell} provide the definition of the classical Bell polynomials, which, in turn, are closely related to the Fa\`a di Bruno formula \cite{riordan}.  

We will show multivariate generalization of the following formulae that express Hopf algebraically the moments/cumulants relations

\begin{theorem}
Setting $\mu,\kappa\in H^*$, $\mu(x^n):=\mu_n$ and $\kappa(x^n):=\kappa_n$, $n\geq 0$, we 
have the relations
\begin{equation}\label{muk0}
  \mu = \exp^\star (\kappa):=\varepsilon+\sum_{n\geq 1} \frac1{n!}\,\kappa^{\star n},
\end{equation}
\begin{equation}\label{kappam0}
	\kappa=\log^\star(\mu):=\sum_{n\geq 1} \frac{(-1)^{n-1}}n \, (\mu- \varepsilon)^{\star n},
\end{equation}
where $\varepsilon(x^k):=\ind{(k=0)}$.
\end{theorem}
Note that the $n$-fold convolution product $\kappa^{\star n}=\kappa \star \cdots \star \kappa$ ($n$ times) is well-defined as the convolution product defined in \eqref{convolproduct} is associative. The above formulae \eqref{muk0} and \eqref{kappam0} are Hopf-algebraic interpretations of
the classical {\it Leonov--Shiryaev relations} \cite{LeonovShiryaev1959}, see \eqref{LS1} 
and \eqref{LS2} below.

\subsection{Deformation of products}

Our Theorem \ref{1.2} above introduces the idea of a deformed product $\cdot_\mu$ in a polynomial 
algebra. This idea is
used in a very important way in the recent theory of regularity structures \cite{reg,bhz,Hai-Cha},
which is based on {\it products of random distributions}, i.e. of generalized functions on $\R^d$. 
Such products are in fact ill-defined and need to be {\it renormalized}; this operation corresponds
algebraically to a deformation of the standard pointwise product, and is achieved through a 
comodule structure which extends the coproduct \eqref{deltan} to a much larger class of generalized
monomials.

In the last sections of this paper we extend the notion of a deformed product to more general 
comodules and we discuss one important and instructive example, the space of decorated rooted trees
endowed with the extraction-contraction operator. This setting is relevant for branched rough paths 
\cite{Gubinelli2010693}, and constitutes a first step towards
the more complex framework of regularity structures \cite{bhz}. 

We hope that this discussion may help the algebraic-minded reader becoming more familiar with 
a theory which combines probability, analysis and algebra in a very deep and innovative way.


\subsection{Organisation of the paper}

In Section \ref{sect:cumulantsHopf} we briefly review classical multivariate moment--cumulants relations. Section \ref{Sect:CumulHopf} provides an interpretation of these relations in a Hopf-algebraic context. In Section \ref{sect:wick} we extend the previous approach to generalized Wick polynomials. Section \ref{sect:deform} is devoted to Hopf algebra deformations, which are applied to  Wick polynomials in Section \ref{sect:anoterway}. In Section \ref{sect:inverse} still another interpretation of Wick polynomials in terms of a suitable comodule structure is introduced. Section \ref{sect:pointwise} explains the deformation of the pointwise product on functions. Section \ref{sect:Wicktrees} addresses the problem of extending our results to Hopf algebras of non-planar decorated rooted trees. It prepares for Section \ref{sect:final}, which outlines briefly the idea of applying the Hopf algebra approach to cumulants and Wick products in the context of regularity structures. 
 
\medskip

Apart from the basic definitions in the theory of coalgebras and Hopf algebras, for which we refer the reader to P.~Cartier's {\sl Primer on Hopf algebras} \cite{Cartier}, this article aims at being a self-contained reference on cumulants and Wick products both for probabilists and algebraists interested in probability. We have therefore detailed proofs and constructions, even those that may 
seem obvious to experts from one of these two fields. 

\medskip
For convenience and in view of applications to scalar real-valued random variables, we fix the field of real numbers $\R$ as ground field. Notice however that algebraic results and constructions in the article depend only on the ground field being of characteristic zero.

\medskip

{\bf{Acknowledgements}}: The second author acknowledges support from the CARMA grant ANR-12-BS01-0017, ``Combinatoire Alg\'ebrique, R\'esurgence, Moules et Applications". The third author was partially supported by the CONICYT/Doctorado Nacional/2013-21130733 doctoral scholarship and acknowledges support from the ``Fondation Sciences Math\'{e}matiques de Paris''. The fourth author acknowledges support of the ANR project ANR-15-CE40-0020-01 grant LSD.


\section{Joint cumulants and moments}
\label{sect:cumulantsHopf}

We start by briefly reviewing classical multivariate moment--cumulant relations. 


\subsection{Cumulants}
\label{ssect:cumulants}

If we have a finite family of random variables $(X_a, a \in \mathcal{S})$ such that $X_a$ has finite moments of all orders for every $a \in \mathcal{S}$, then the analogue of the exponential formula \eqref{expo} holds
\begin{equation}
\label{expo2}
	\exp\left(\sum_{\underbar{n}\in \N^\mathcal{S}} \frac{t^{\underbar{n}}}{\underbar{n}!}\,\kappa_{\underbar{n}}\right) 
	= \sum_{\underbar{n}\in\N^\mathcal{S}} \frac{t^{\underbar{n}}}{\underbar{n}!}\, \mu_{\underbar{n}},
\end{equation}
where $\mu_{\underbar{n}}:=\E(X^{\underbar{n}}).$ Here we use multivariable notation, i.e., with $\N:=\{0,1,2,3,\ldots\}$ and $(t_a,a\in \mathcal{S})$ commuting variables, we define for $\underbar{n}=(n_a,a\in \mathcal{S})\in\N^\mathcal{S}$
\[
	t^{\underbar{n}}:=\prod_{a\in \mathcal{S}} (t_a)^{n_a}, 
	\qquad 
	X^{\underbar{n}}:=\prod_{a\in \mathcal{S}} (X_a)^{n_a}, 
	\qquad 
	\underbar{n}!:=\prod_{a\in \mathcal{S}}(n_a)!,
\]
and we use the conventions $(t_a)^0:=1$, $(X_a)^0:=1$. This defines in a unique way the family $(\kappa_{\underbar{n}}, \underbar{n} \in \N^\mathcal{S})$ of joint cumulants of $(X_a,a\in \mathcal{S})$ once the family of corresponding joint moments $(\mu_{\underbar{n}}, \in\N^\mathcal{S})$ is given. When it is necessary to specify the dependence of $\kappa_{\underbar{n}}$ on $(X_a,a\in \mathcal{S})$ we shall write $\kappa_{\underbar{n}}(X)$, and similarly for $\mu_{\underbar{n}}$.

Identifying a subset $B \subseteq \mathcal{S}$ with its indicator function $\ind{B} \in \{0,1\}^\mathcal{S}$, we can use the notation $ \kappa_B$ and $\mu_B$ for the corresponding joint cumulants and moments. The families $(\kappa_B, B\subseteq \mathcal{S})$ and  $(\mu_B, B\subseteq \mathcal{S})$ satisfy the so-called {\it Leonov--Shiryaev relations} \cite{LeonovShiryaev1959,Speed}
\begin{align}
\label{LS1}
	\mu_B&=\sum_{\pi\in\CP(B)} \prod_{C\in\pi}\kappa_{C}\\
\label{LS2}
	\kappa_B&=\sum_{\pi\in \CP(B)} (|\pi|-1)! \, (-1)^{|\pi|-1} \prod_{{C}\in\pi} \mu_{C} ,
\end{align}
where we write $\CP(B)$ for the set of all set partitions of $B$, namely, all collections $\pi$ of subsets (blocks) of $B$ such that $\cup_{{C} \in \pi} {C}=B$ and elements of $\pi$ are pairwise disjoint; moreover $|\pi|$ denotes the number of blocks of $\pi$, which is finite since $B$ is finite. Formulae \eqref{LS1} and \eqref{LS2} have been intensively studied from a combinatorial perspective, see, e.g., \cite[Chapter 2]{Pec-Taq}. Regarding the properties of cumulants we refer the reader to \cite{Speed}. 

Formula \eqref{LS1} has in fact been adopted, for instance, in \cite{Hai-She} as a recursive definition for $(\kappa_B, B\subseteq \mathcal{S})$. This approach does indeed determine the cumulants uniquely by induction over the cardinality $|B|$ of the finite set $B$. This follows from the right-hand side containing $\kappa_B$, which is what we want to define, as well as $\kappa_{C}$ for some ${C}$ with $|{C}| < |B|$, which have been already defined in lower order.

Although this recursive approach seems less general than the one via exponential generating functions as in \eqref{expo2}, since it forces to consider only $\underbar{n} \in \{0,1\}^\mathcal{S}$, it turns out that they are equivalent. Indeed, replacing $(X_a, a \in \mathcal{S})$ with $(Y_b, b \in \mathcal{S}\times\N^\ast)$, where $Y_b:=X_a$ for $b=(a,k)\in \mathcal{S}\times\N^\ast$, then for $\underbar{n}\in\N^\mathcal{S}$ we have
\[
	\kappa_{\underbar{n}}(X)=\kappa_B(Y), 
	\qquad 
	B=\{(a,k): a\in \mathcal{S}, \ 1\leq k\leq \underbar{n}(a)\}.
\]

In this paper we show that the Leonov--Shiryaev relations \eqref{LS1}-\eqref{LS2} have an elegant Hopf-algebraic interpretation which also extends to Wick polynomials. Notice that a different algebraic interpretation of \eqref{LS1}-\eqref{LS2} has been given in terms of M\"obius calculus \cite{Pec-Taq,Speed}. Moreover, the idea of writing moment--cumulant relations in terms of convolution products is closely related to Rota's Umbral calculus \cite{JoRo,Rota-Shen}.


\section{From cumulants to Hopf algebras}
\label{Sect:CumulHopf}

In this section we explain how classical moment--cumulant relations can be encoded using Hopf algebra techniques. These results may be folklore among followers of Rota's combinatorial approach to probability, and, as we already alluded at, there exist actually in the literature already various other algebraic descriptions of moment--cumulant relations (via generating series as well as more sophisticated approaches in terms of umbral calculus, tensor algebras and set partitions). Our approach is most suited regarding our later applications, i.e., the Hopf algebraic study of Wick products. Since these ideas do not seem to be well-known to probabilists, we believe that they deserve a detailed presentation.


\subsection{Moment--cumulant relations via multisets}
\label{ssect:multisets}

Throughout the paper we consider a fixed collection of real-valued random variables $\CX = \{X_a\}_{a \in\CA}$ defined on a probability space $(\Omega,{\mathcal F},{\mathbb P})$ for an index set $\CA$. We suppose that $X_a$ has finite moments of all orders for every $a\in\CA$.

We do not assume that $\CA$ is finite, but moments and cumulants will be defined only for finite subfamilies. We extend the setting of \eqref{expo2}, where $\mathcal{S}$ was a finite set, by defining $\MS(\CA) \subset \N^\CA$ as the set of all {\it finitely supported} functions $B:\CA\to\N$. In the case of $B\in\MS(\CA)\cap\{0,1\}^\CA$, we have that $B$ is the indicator function of a finite set $\mathcal{S}(B)$, namely the support of $B$. For a general $C\in\MS(\CA)$, we can identify the finite set $\mathcal{S}(C)$ given by the support of $C$, and then $C(a) \geq 1$ can be interpreted as the multiplicity of $a\in \mathcal{S}(C)$ in $C$ viewed as a multiset.

This multiset context is motivated by the following natural definition for $B\in\MS(\CA)$:
\begin{equation}\label{def:X_B}
	X^\emptyset:=1, \qquad
  X^B := \prod_{\substack{a\in \CA \\ B(a)>0}} (X_a)^{B(a)}.
\end{equation}
For all $B\in\MS(\CA)$ we also set
\[
	|B|:=\sum_{a\in \CA} B(a)<+\infty.
\]

The set $\MS(\CA)$ is a poset for the partial order defined by $B\leq B'$ if and only if $B(a)\leq B'(a)$ for all $a$ in $\CA$. 
Moreover, it is a commutative monoid for the product 
\begin{equation}\label{cdot}
	(A \cdot B)(a):=A(a)+B(a), \qquad a\in\CA,
\end{equation}
for $A,B$ in $\MS(\CA)$, i.e., the map $(A,B)\to A\cdot B$ is associative and commutative. 
The set $\MS(\CA)$ is actually the free commutative monoid generated by the indicator functions of the one-element sets $\{a\},\ a\in\CA$ (with neutral element the indicator function of the empty set). 

\begin{definition}\label{def:H}
We call $H$ the vector space freely generated by the set $\MS(\CA)$.
\end{definition}

As a vector space, $H$ is isomorphic to the algebra of polynomials over the set of (commuting) variables $x_a, \ a\in\CA$ (the isomorphism given by mapping $B\in\MS(\CA)$ to monomials $\prod\limits_{a \in \CA}x_a^{B(a)}$). Moreover the product \eqref{cdot} is motivated, using the notation \eqref{def:X_B}, by
\[
	X^{A\cdot B} = X^A X^B, \qquad A,B \in \MS(\CA),
\]
and is therefore the multivariate analogue of \eqref{xn}.

For $n\geq 1$ and for $B,B_1,\ldots,B_n\in\MS(\CA)$  we set
\[
	\binom{B}{B_1 \dots B_n}:=\ind{(B_1\cdot B_2 \cdots B_n=B)}\prod_{a\in\CA} \frac{B(a)!}{B_1(a)!\cdots B_n(a)!},
\]
where $B_1\cdot B_2 \cdots B_n$ is the product of $B_1,\ldots,B_n$ in $\MS(\CA)$ using the
multiplication law defined in \eqref{cdot}. Note that for a given $B\in\MS(\CA)$, there exist only finitely many 
$B_1,\ldots,B_n\in\MS(\CA)$ such that $B_1\cdot B_2 \cdots B_n=B$.

\begin{definition} \label{def:cumu}
For every $B\in\MS(\CA)$, we define the cumulant $\Cum(X_B)$ inductively over $|B|$ by $\Cum(X_\emptyset)=0$ and else
\begin{equation} \label{LS1'}
	\E \big( X^B\big)
	= \sum\limits_{n=1}^{|B|}\frac{1}{n!}\sum_{B_1,\ldots,B_n\in\MS(\CA)\setminus\{\emptyset\}} 
	\binom{B}{B_1\dots B_n}\prod_{i=1}^n \Cum \big(X_{B_i}\big)\;.
\end{equation}
\end{definition}

\begin{remark} If $B\in\MS(\CA)\cap\{0,1\}^\CA$, then \eqref{LS1'} reduces to the first Leonov--Shiryaev relation \eqref{LS1}, since on the right-hand side of \eqref{LS1'} $B_1,\ldots,B_n\in\MS(\CA)$ are also in $\{0,1\}^\CA$ and in particular the binomial coefficient (when non-zero) is equal to 1.
\end{remark}

As we will show in \eqref{expo3} below, expression \eqref{LS1'} is equivalent to the usual formal power series definition of cumulants (whose exponential generating series is the logarithm of the exponential generating series of moments). As for \eqref{LS1}, expression \eqref{LS1'} does indeed determine the cumulants uniquely by induction over $|B|$. This is because the right-hand  side only involves $\Cum(X_B)$, which is what we want to define, as well as $\Cum(X_{\bar B})$ for some $\bar B$ with $|\bar B| < |B|$, which is already defined  by the inductive hypothesis.

\subsection{Exponential generating functions}

Define two linear functionals on $H$:
\begin{equation}\label{muka}
	\begin{array}{ll}
	\mu: &H\to\R \\ 
	       & A \mapsto \mu(A):=\E(X^A)
	\end{array}
	\qquad\ 
	\begin{array}{ll}
	\kappa: & H\to\R\\ 
		   & A \mapsto \kappa(A):=\Cum(X_A),\\
	\end{array}
\end{equation}
where $A\in \MS(\CA)$,
 $\mu(\emptyset):=1$ and $\kappa(\emptyset):=0$.

Let us fix a finite subset $\mathcal{S}=\{a_1,\ldots,a_p\}\subset\CA$. For $B\in\MS(\mathcal{S})$ we set 
\[
	t^B:=\prod_{i=1}^p (t_i)^{B(a_i)},
\]
where the $t_i$ are commuting variables. Then we define for $B\in\MS(\CA)$ the factorial 
\[
	B!:=\prod_{a\in\CA} (B(a))!,
\]
and the exponential generating function of $\lambda\in H^*$ (seen as a formal power series in the variables $t_i$)
\[
	\varphi_\lambda(t,\mathcal{S}):=\sum_{B\in\MS(\mathcal{S})} \frac{t^B}{B!} \,\lambda(B).
\]
Then from Definition \ref{def:cumu} we get the usual exponential relation between the exponential moment and cumulant generating functions of $\mu$ and $\kappa$, analogous to \eqref{expo} and \eqref{expo2}: 
\begin{equation}\label{expo3}
	\begin{split}
	\varphi_\mu(t,\mathcal{S})&=\sum_{B\in\MS(\mathcal{S})} \frac{t^B}{B!}\, \mu(B)
	\\ & = \sum_{n\geq 0} \frac1{n!}\sum_{B\in\MS(\mathcal{S})} \sum_{B_1,\ldots,B_n\in\MS(\mathcal{S})} 
	\frac1{B!}\binom{B}{B_1\dots B_n}
	\prod_{i=1}^n\left({t^{B_i}} \,\kappa(B_i)\right)
	\\ & = \sum_{n\geq 0} \frac1{n!}\sum_{B_1,\ldots,B_n\in\MS(\mathcal{S})} \prod_{i=1}^n\left(\frac{t^{B_i}}{B_i!} \,\kappa(B_i)\right)
	\\ & = \sum_{n\geq 0} \frac1{n!}\left(\sum_{B\in\MS(\mathcal{S})} \frac{t^B}{B!}\, \kappa(B)\right)^n
	=\exp(\varphi_\kappa(t,\mathcal{S})).
\end{split}
\end{equation}
From \eqref{expo3} we obtain another recursive relation between moments and cumulants. Let us set $\ind{(a)}(b):=\ind{(a=b)}$ for $a,b\in\CA$. Then we have
\begin{equation}
\label{otrec}
	\mu(A\cdot\ind{(a)}) = \sum_{B_1,B_2\in\MS(\CA)} \binom{A}{B_1 \,B_2} \, \kappa(B_1\cdot\ind{(a)})\, \mu(B_2).
\end{equation}
This recursion is the multivariate analogue of the one in \eqref{Bell}.


\subsection{Moment--cumulant relations and Hopf algebras}
\label{ssect{MCHopf}}

We endow now the space $H$ from Definition \ref{def:H} with the commutative and associative product $\cdot:H\otimes H\to H$ induced by the monoid structure of $\MS(\CA)$ defined in \eqref{cdot}.  
The unit element is the null function $\emptyset$ (we will also write abusively $\emptyset$ for the unit map -- the embedding of $\R$ into $H$: $\lambda\longmapsto \lambda\cdot\emptyset$). We also define a coproduct $\Delta: H\to H\otimes H$ on $H$ by
\begin{equation}
\label{coprodH}
	\Delta A:=\sum_{B_1,B_2\in\MS(\CA)} \binom{A}{B_1\, B_2} \left[B_1\otimes B_2\right],
\end{equation}
recall \eqref{deltan}.
The counit $\varepsilon:H\to\R$ is defined by $\varepsilon(A)=\ind{(A=\emptyset)}$ and turns $H$ into a coassociative counital coalgebra. Coassociativity $(\Delta\otimes \id) \Delta=(\id\otimes\Delta ) \Delta$ follows from the associativity of the monoid $\MS(\CA)$:
\begin{align*}
	(\Delta\otimes \id) \Delta A&=
	\sum_{(B_1\cdot B_2)\cdot B_3=A} \binom{A}{B_1B_2B_3} \left[B_1\otimes B_2\otimes B_3\right]\\
	&=\sum_{B_1\cdot (B_2\cdot B_3)=A} \binom{A}{B_1B_2B_3} \left[B_1\otimes B_2\otimes B_3\right]\\
	&=(\id\otimes\Delta ) \Delta A.
\end{align*}

\begin{proposition}\label{prop:Hopf}
$H$ is a commutative and cocommutative bialgebra and, since $H$ is graded by $|A|$ as well as connected, it is a Hopf algebra. 
\end{proposition}

\begin{proof}
Indeed, we already noticed that $H$ is isomorphic as a vector space to the polynomial algebra, denoted $P$, generated by commuting variables $x_a,\ a\in\CA$. The latter is uniquely equipped with a bialgebra and Hopf algebra structure by requiring the $x_a$ to be primitive elements, that is, by defining a coproduct map $\Delta_P: P\to P\otimes P$ such that it is an algebra map and $\Delta_P(x_a)=x_a\otimes 1+1\otimes x_a$. Recall that since $P$ is a polynomial algebra, these two conditions define $\Delta_P$ uniquely. The antipode is the algebra endomorphism of $P$ induced by $S(x_a):=-x_a$. We let the reader check that the natural isomorphism between $H$ and $P$ is an isomorphism of algebras and maps $\Delta$ to $\Delta_P$. The proposition follows.
\end{proof}

Recall that the dual of a coalgebra is an algebra, which is associative (resp.~unital, commutative) if the coalgebra is coassociative (resp.~counital, cocommutative). In particular, the dual $H^\ast$ of $H$ is equipped with an associative and commutative unital product written $\star$, defined  for all $f,g\in H^\ast$ and $A\in H$ by:
\begin{equation}
\label{convol}
	(f\star g)(A):=(f\otimes g)\Delta A.
\end{equation}
The unit of this product is the augmentation map $\varepsilon$. For later use, we also mention that the associative product defined in \eqref{convol} extends to linear endomorphisms $f,g$ of $H$ as well as to the product of a linear form on $H$ with a linear endomorphism of $H$, by the same defining formula.

We denote $\Delta^0 :=\id:H\to H$, $\Delta^1:=\Delta:H\to H\otimes H$, and for $n\geq 2$:
\[
	\Delta^n:=(\Delta\otimes\id)\Delta^{n-1}:H\to H^{\otimes (n+1)}.
\]

\begin{proposition}\label{mukappa}
We have
\begin{equation}\label{muk}
	\mu = \exp^\star (\kappa)=\varepsilon+\sum_{n\geq 1} \frac1{n!}\,\kappa^{\star n}.
\end{equation}
\end{proposition}

\begin{proof} By the definitions of $\Delta$ and $\Delta^n$ we find that
\[
\begin{split}
	\kappa^{\otimes n}\Delta^{n-1}A& = \sum_{B_1,\ldots,B_{n}\in\MS(\CA)} \binom{A}{B_1\dots B_n}\, \prod_{i=1}^n \kappa(B_i),
\end{split}
\]
and, by  \eqref{LS1'}, this yields the result.
\end{proof}

\begin{remark}
Formula \eqref{muk} is the Hopf-algebraic analogue of the
first Leo\-nov-Shiryaev relation \eqref{LS1}.
\end{remark} 

Since $\kappa(\emptyset)=0$, $\kappa^{\otimes n}\Delta^{n-1}(A)$ vanishes whenever $|A|>n$. Similarly, under the same assumption, $(\mu-\varepsilon)^{\star n}(A)=0$. It follows that one can handle formal series identities such as $\log^\star (\exp^\star (\kappa))=\kappa$ or $\exp^\star (\log^\star (\mu))=\mu$ without facing convergence issues. In particular

\begin{proposition}\label{kappamu}
We have
\begin{equation}\label{kappam}
	\kappa=\log^\star(\mu)=\sum_{n\geq 1} \frac{(-1)^{n-1}}n \, (\mu- \varepsilon)^{\star n}.
\end{equation}
\end{proposition}

From Proposition \ref{kappamu} we obtain the formula

\begin{equation} \label{LS2'}
	\Cum(X_B) = \sum_{n\geq 1} \frac{(-1)^{n-1}}n \sum_{B_1,\ldots,B_n\in\MS(\CA)\setminus
	\{\emptyset\}} 
	\binom{B}{B_1\dots B_n}\, \prod_{i=1}^n \E(X^{B_i})
\end{equation}
which may be considered the inverse to \eqref{LS1'}. 

\begin{remark}
The formula \eqref{kappam} is the Hopf-algebraic analogue of the second Leonov--Shiryaev relation \eqref{LS2}. Moreover, for $B\in\MS(\CA)\cap\{0,1\}^\CA$, then \eqref{LS2'} 
also reduces to the second Leonov--Shiryaev relation \eqref{LS2}, since on the right-hand side of \eqref{LS2'} $B_1,\ldots,B_n$ $\in\MS(\CA)$ are also in $\{0,1\}^\CA$. 
\end{remark}


\subsection{A sub-coalgebra}
\label{ssect:subsets}

If one prefers to work in the combinatorial framework of the Leonov--Shiryaev formulae \eqref{LS1}-\eqref{LS2} rather than with \eqref{LS1'}-\eqref{LS2'}, then one may consider the linear span $J$ of $\MS(\CA)\cap\{0,1\}^\CA$ (namely of all finite subsets of 
$\CA$, or of their indicator functions). 

Then $J$ is a linear subspace of $H$, which is {\it not} a sub-algebra of $H$ for the product $\cdot$ defined in \eqref{cdot}. The coproduct $\Delta$
defined in \eqref{coprodH} coacts however nicely on $J$ since for all finite subsets $A$ of $\CA$
\[
	\Delta A=\sum_{B_1\cdot B_2=A} B_1\otimes B_2 \in J\otimes J.
\]
Moreover the restriction of $\varepsilon$ to $J$ defines a counit for $(J,\Delta)$. Therefore $J$ is a sub-coalgebra of $H$. With a slight abuse of notation we still write $\star$ for the dual product on $J^\ast$
\[
	(f\star g)(A):=(f\otimes g)\Delta A,
\]
for $A\in J$ and $f,g\in J^\ast$. If we denote as before
$$
	\begin{array}{ll}
	\mu: &J\to\R \\ 
	       & A \mapsto \mu(A):=\E(X^A)
	\end{array}
	\qquad\ 
	\begin{array}{ll}
	\kappa: &J\to\R\\ 
		   & A \mapsto \kappa(A):=\Cum(X_A),\\
	\end{array}
$$
with $A\in \MS(\CA)\cap\{0,1\}^\CA$, $\mu(\emptyset):=1$ and $\kappa(\emptyset):=0$, then the Leonov--Shiryaev relations \eqref{LS1}-\eqref{LS2} can be rewritten in $J^\ast$ as, respectively,
\[
	\mu = \exp^\star (\kappa)=\varepsilon+\sum_{n\geq 1} \frac1{n!}\,\kappa^{\star n}
\]
and
\[
	\kappa=\log^\star(\mu)=\sum_{n\geq 1} \frac{(-1)^{n-1}}n \, (\mu- \varepsilon)^{\star n}. 
\]


\section{Wick Products}
\label{sect:wick}

The theory of Wick products, as well as the related notion of chaos decomposition, play an important role in various fields of applied probability. Both have deep structural features in relation to the fine structure of the algebra of square integrable functions associated to one or several random variables. The aim of this section and the following ones is to revisit the theory on Hopf algebraic grounds. The basic observation is that the formula for the Wick product is closely related to the recursive definition of antipode in a connected graded Hopf algebra. This approach seems to be new, also from the point of view of concurring approaches such as umbral calculus \cite{JoRo,Rota-Shen} or set partition combinatorics \`a la Rota--Wallstrom \cite{RoWa}.


\subsection{Wick polynomials}
\label{ssect:Wick}

We are going to use extensively the notion of Wick polynomials for a collection of (not necessarily Gaussian) random variables 
which is defined as follows. 
 
\begin{definition} \label{def:Wick-product}
Given a collection $\CX = \{X_a\}_{a \in \CA}$ of random variables with finite moments of 
all orders, for any $A \in \MS(\CA)$ the Wick polynomial $\Wick{X_A}$ is a random variable 
defined recursively by setting $\Wick{X_\emptyset} = 1$ 
and postulating that
\begin{equation} 
\label{defWick}
	X^A := \sum_{B_1,B_2\in\MS(\CA)} \binom{A}{B_1\,B_2}\, \E(X^{B_1})  \Wick{X_{B_2}}.
\end{equation}
\end{definition}

As for cumulants, \eqref{defWick} is sufficient to define $\Wick{X_A}$ by recursion over $|A|$. Indeed, the term with $B_2=A$ is precisely the quantity we want to define, and all other terms only involve Wick polynomials $\Wick{X_B}$, for $B\in\MS(\CA)$ with $|B|<|A|$.

It is now clear that formula \eqref{defWick} can be lifted to $H$ as 
\begin{equation} 
\label{defWick2}
	A  = \sum_{B_1,B_2\in\MS(\CA)} \binom{A}{B_1\,B_2}\, \E(X^{B_1})  \Wick{{B_2}},
\end{equation}
and written in Hopf algebraic terms as follows
\begin{equation}
\label{WickCop}
	A  = (\mu\star W) (A) = (\mu\otimes W)\Delta A, 
\end{equation}
for $A \in\MS(\CA)$. We have set $W: H\to H$, $W(A):=\Wick{A}$ and call $W$ the Wick product map (see Theorem \ref{ww} for a justification of the terminology). Notice that it depends on the joint distribution of the $X_a$s. Formula \ref{WickCop} is the Hopf algebraic analogue of the definition of the Wick polynomial $\Wick{X_B}$ used in references \cite{Hai-She,Luk-Marc}. Moreover, introducing the algebra map $\ev:A\longmapsto X^A$ from $H$ to the algebra of random variables generated by $(X_a,a\in\CA)$, one gets by a recursion over $|A|$ that $\ev(\Wick{A})=\Wick{X_A}$ (for that reason, from now on we will call slightly abusively both $\Wick{A}$ and the random variable $\Wick{X_A}$ the Wick polynomial associated to $A$).


\subsection{A Hopf algebraic construction}
\label{ssect:2Hopf}

We want to present now a closed Hopf algebraic formula for the Wick polynomials introduced in Definition \ref{def:Wick-product}. We define the set ${\mathcal G}(H):=\{\lambda\in H^*: \lambda(\emptyset)=1\}$. Then it is well known that ${\mathcal G}(H)$ is a group for the $\star$-product. Indeed, any $\lambda\in{\mathcal G}(H)$ has an inverse $\lambda^{-1}$ in ${\mathcal G}(H)$ given by the Neumann series
\begin{equation}
\label{inverse1}
	\lambda^{-1}=\sum_{n\geq 0} (\varepsilon-\lambda)^{\star n}.
\end{equation}
As usual, this infinite sum defines an element of $H^*$ since, evaluated on any $A \in H$, it reduces to a finite number of terms. 

\begin{theorem}\label{wick:repr}
Let $\mu\in{\mathcal G}(H)$ be given by $\mu(A)=\E(X^A)$, then for all $A\in\MS(\CA)$
\begin{equation}\label{W}
	\Wick{A}=W(A)=(\mu^{-1}\star\id)(A)=(\mu^{-1}\otimes\id)\Delta A.
\end{equation}
\end{theorem}
\begin{proof}
The identity follows from \eqref{WickCop} and from the associativity of the $\star$ product.
\end{proof}

From \eqref{inverse1} and \eqref{W} we obtain

\begin{proposition}\label{rec} Wick polynomials have the explicit expansion
\[
\begin{split}
	\Wick{A}\!=\! A \!+\! \sum_{n\geq 1} (-1)^n \sum_{B\in\MS(\CA)} 
  \sum_{\substack{B_1,\ldots,B_{n}\in\MS(\CA) \\ B_i \neq \{\emptyset\}}} 
	\binom{A}{B_1\dots B_n B} \, \mu(B_1)\cdots\mu(B_n)\,  B.
\end{split}
\]
\end{proposition}


\section{Hopf algebra deformations}
\label{sect:deform}

The group ${\mathcal G}(H)=\{\lambda\in H^*: \lambda(\emptyset)=1\}$ equipped with the $\star$ product acts canonically on $H$ by means of the map $\phi_\lambda:H\to H$ 
\begin{equation}\label{phi}
	\phi_\lambda(A):=(\lambda\otimes\id)\Delta A, 
\end{equation}
for $\lambda\in{\mathcal G}(H)$ and $A \in H$. In other words, $\phi_\lambda=\lambda\star\id=\id\star\lambda$, the latter identity following from cocommutativity of $\Delta$. This is a group action since one checks easily using the coassociativity of $\Delta$ that
\[
	\phi_{\lambda_1\star\lambda_2}=\phi_{\lambda_1}\circ\phi_{\lambda_2}, 
\]
so that in particular
\[
	(\phi_{\lambda})^{-1}=\phi_{\lambda^{-1}}.
\]
Being invertible, the maps $\phi_\lambda$ allow to define {\it deformations} of the product $\cdot$ defined in \eqref{cdot}, as well as of the coproduct $\Delta$ defined in \eqref{coprodH} and of the counit. Namely we define $\cdot_\lambda: H\otimes H\to H$, $\Delta_\lambda:H\to H\otimes H$ and $\varepsilon_\lambda$ by
\begin{equation}\label{eq:deformation}
\begin{split}
	A\cdot_\lambda B \ & :=\phi_\lambda^{-1}(\phi_\lambda(A)\cdot\phi_\lambda(B)), 
\\
	\Delta_\lambda A \ &
			:=(\phi_\lambda^{-1}\otimes \phi_\lambda^{-1})\Delta\phi_\lambda A,
\\
	\varepsilon_\lambda (A) \ & :=\varepsilon\circ \phi_\lambda(A)=\lambda(A).
\end{split}
\end{equation}
Although $\varepsilon_\lambda=\lambda$, we find it useful to introduce the notation $\varepsilon_\lambda$ to feature the new role of $\lambda$ as a counit.

Notice that, as $\lambda(\emptyset)=1$, we have $\phi_\lambda(\emptyset )=\emptyset$ and $A\cdot_\lambda\emptyset =A$. Dually, 
$$
	(\varepsilon_\lambda\otimes\id)\circ\Delta_\lambda(A)=(\varepsilon\otimes\phi_{\lambda^{-1}})\Delta\phi_{\lambda}(A)=A.
$$
Then we have

\begin{theorem}\label{deformation}
For any $\lambda\in{\mathcal G}(H)$, the quintuple $(H,\cdot_\lambda,\emptyset,\Delta_\lambda,\varepsilon_\lambda)$ defines a Hopf algebra. The map 
\[
	\phi_\lambda^{-1}:(H,\cdot,\emptyset,\Delta,\varepsilon)\to
	(H,\cdot_\lambda,\emptyset,\Delta_\lambda,\varepsilon_\lambda)
\]
is an isomorphism of Hopf algebras.
\end{theorem}

\begin{proof}Although the Theorem follows directly from the properties of conjugacy, we detail the proof.
Associativity of $\cdot_\lambda$ and coassociativity of $\Delta_\lambda$ follow directly. First,
\[
	(A\cdot_\lambda B)\cdot_\lambda C 
	= \phi_\lambda^{-1}(\phi_\lambda(A)\cdot\phi_\lambda(B)\cdot\phi_\lambda(C))
	=A\cdot_\lambda (B\cdot_\lambda C),
\]
which shows associativity. Coassociativity is simple to see as well 
\[
\begin{split}
	(\Delta_\lambda\otimes\id)\Delta_\lambda A
	&=(\phi_\lambda^{-1}\otimes \phi_\lambda^{-1}\otimes \phi_\lambda^{-1})(\Delta\otimes\id)\Delta\phi_\lambda A\\
	& = (\phi_\lambda^{-1}\otimes \phi_\lambda^{-1}\otimes \phi_\lambda^{-1})(\id\otimes\Delta)\Delta\phi_\lambda A \\
	&= (\id\otimes\Delta_\lambda)\Delta_\lambda A.
\end{split}
\]
We check now the compatibility relation between $\cdot_\lambda$ and $\Delta_\lambda$: 
\[
\begin{split}
	(\Delta_\lambda A)\cdot_\lambda(\Delta_\lambda B) 
	& = (\phi_\lambda^{-1}\otimes\phi_\lambda^{-1}) \Big(
	\big((\phi_\lambda\otimes\phi_\lambda)\Delta_\lambda A\big)\cdot
	\big((\phi_\lambda\otimes\phi_\lambda)\Delta_\lambda B\big)\Big)
	\\ & = (\phi_\lambda^{-1}\otimes\phi_\lambda^{-1}) 
	\big((\Delta\phi_\lambda A) \cdot (\Delta\phi_\lambda B)\big)
	\\ & = (\phi_\lambda^{-1}\otimes\phi_\lambda^{-1}) 
	\Delta\left(\phi_\lambda A \cdot \phi_\lambda B\right)
	\\ &= \Delta_\lambda(A\cdot_\lambda B).
\end{split}
\]
Finally, we check that $\phi_\lambda^{-1}:(H,\cdot,\emptyset, \Delta,\varepsilon)\to(H,\cdot_\lambda,\emptyset,\Delta_\lambda,\varepsilon_\lambda)$ is a bialgebra morphism:
\[
	\phi_\lambda^{-1}(A\cdot B)=\phi_\lambda^{-1}(A)\cdot_\lambda\phi_\lambda^{-1}(B),
\]
\[
	(\phi_\lambda^{-1}\otimes\phi_\lambda^{-1})\Delta A = \Delta_\lambda\phi_\lambda^{-1} A.
\]

We have proved until now that $\phi_\lambda^{-1}$ is a isomorphism of bialgebras. Since $(H,\cdot_\lambda,\emptyset,\Delta_\lambda,\varepsilon_\lambda)$ is a graded connected bialgebra, it has an antipode. Since moreover the 
antipode of a Hopf algebra is unique, we obtain that $\phi_\lambda^{-1}$ preserves the antipode as well.
\end{proof}

\begin{remark}\label{import}{\rm 
The construction of \eqref{eq:deformation} and Theorem \ref{deformation} works also if we 
replace $\phi_\lambda$ with any linear invertible 
map $\phi:H\to H$ such that $\phi(\emptyset)=\emptyset$. Indeed, in the above 
considerations we have never used the formula \eqref{phi} which defines $\phi_\lambda$.
}
\end{remark}

In the particular case of $\lambda=\mu$, where $\mu$ is the moment functional defined in \eqref{muka}, we obtain by Theorem \ref{wick:repr} and Theorem \ref{deformation}:

\begin{theorem}\label{ww}
The Wick product map $W(A)=\Wick{A}$ is equal to $\phi_{\mu^{-1}}$. Therefore $W:(H,\cdot,\emptyset,\Delta,\varepsilon)\to(H,\cdot_{\mu},\emptyset,\Delta_{\mu},\varepsilon_\mu)$ is a Hopf algebra isomorphism, in particular
\[
	\Wick{A_1\cdot A_2} = \Wick{A_1} \cdot_\mu \Wick{A_2}, 
\]
for $A_1,A_2\in H$.
\end{theorem}
More generally, we obtain for any $A_1,\ldots,A_n\in H$  that 
\begin{equation}\label{wickdef}
	\Wick{A_1\cdots A_n} = \Wick{A_1} \cdot_\mu\  \cdots \ \cdot_\mu \Wick{A_n}. 
\end{equation}

\medskip
We notice at last an interesting additional  result expressing abstractly compatibility relations between the two Hopf algebra structures on $H$ (see also Proposition \ref{com2} below). We recall that a linear space $M$ is a left comodule over the coalgebra $(H,\Delta,\varepsilon)$ if there is linear map $\rho:M\to H\otimes M$ such that
\begin{equation}\label{deltarho}
	(\Delta\otimes\id_M)\rho
	=(\id_H\otimes\rho)\rho, \qquad (\varepsilon\otimes\id_M)\rho
	= \id_M.
\end{equation}
A left comodule endomorphism of $M$ is then a linear map $f:M\to M$ such that
\[
	\rho \circ f=(\id_H\otimes f)\rho.
\]
In particular the coalgebra $(H,\Delta,\varepsilon)$ is a left comodule over itself, with $\rho=\Delta$. 

\begin{proposition}\label{com1}
If we consider $H$ as a left comodule over itself, then $\phi_\lambda$ is a left comodule morphism for all linear $\lambda:H\to\R$, namely
\begin{equation}\label{deltapartial}
	\Delta\phi_\lambda = (\id\otimes\phi_\lambda)\Delta.
\end{equation}
In particular the Wick product map $W$ is a left comodule endomorphism of $(H,\Delta,\varepsilon)$.
\end{proposition}
\begin{proof}
We have
\[
\begin{split}
\Delta\phi_\lambda & = (\lambda\otimes\id\otimes\id)(\id\otimes\Delta)\Delta 
\\ & = (\lambda\otimes\id\otimes\id)(\Delta\otimes\id)\Delta 
\\ & = (\id\otimes\lambda\otimes\id)(\Delta\otimes\id)\Delta 
\\ & = (\id\otimes\lambda\otimes\id)(\id\otimes\Delta)\Delta 
\\ & = (\id\otimes\phi_\lambda)\Delta,
\end{split}
\]
where we have used, in this order, coassociativity, cocommutativity and then 
coassociativity again. 
\end{proof}


\section{Wick products as Hopf algebra deformations}
\label{sect:anoterway}

Let $a\in\CA$. We define now the functional $\zeta_a:H\to \R$ given by $\zeta_a(A):=\ind{(A=\{a\})}$, for every $A\in\MS(\CA)$. Then we define the operator $\partial_a:H\to H$ as $\partial_a :=\zeta_a\star\id=\phi_{\zeta_a}$ in the notation \eqref{phi}, namely
\[
	\partial_a A=(\zeta_a\otimes\id)\Delta A.
\]
It is simple to see that $\partial_a$ acts as a formal partial derivation with respect to 
$a$, namely it satisfies for $A,B\in\MS(\CA)$ and $a,b\in\CA$
\[
	\partial_a \{b\}=\ind{(a=b)}\emptyset, \qquad 
	\partial_a(A\cdot B)=\partial_a(A)\cdot B+A\cdot \partial_a(B),
\]
since $\zeta_a$ satisfies $\zeta_a(\emptyset)=0$ and $\zeta_a(A\cdot B)=\zeta_a(A)\varepsilon(B)+\varepsilon(A)\zeta_a(B)$, namely $\zeta_a$ is an infinitesimal character. 
Recall that the product $A \cdot B$ has been defined in \eqref{cdot} and that $\{b\}$ is identified with
$\ind{\{b\}}\in H$.

Then the following result is a reformulation in our setting of \cite[Proposition 3.4]{Luk-Marc}.
\begin{theorem}\label{other}
The family of polynomials $(\!\Wick{A}, A\in\MS(\CA))$ is the only collection such that $\Wick{{\emptyset}} = \emptyset$ and for all non-null $A\in\MS(\CA)$ and $a\in\CA$
\begin{equation} 
\label{defWick3}
	{\partial_a} \Wick{ A} = \ \Wick{{\partial_a} A} 
	\quad \mathrm{and} \quad
	\mu(\!\Wick{A}\!)=0.
\end{equation}
\end{theorem}

\begin{proof}
Since $\mu \in {\mathcal G}(H)$ equation \eqref{W} implies
\[
	\mu(\!\Wick{A}\!) = (\mu^{-1}\star\mu)(A)=\varepsilon(A)=\ind{(A=\emptyset)}.
\]
Using \eqref{deltapartial} for $\lambda=\zeta_a$ we obtain
\[
	\Delta\partial_a = (\id\otimes\partial_a)\Delta.
\]
We conclude from \eqref{W} that 
\[
{\partial_a} \Wick{ A} = (\zeta_a\star \mu^{-1}\star\id )(A) = 
(\mu^{-1}\star\zeta_a\star\id) (A) = \Wick{{\partial_a} A} 
\]
by the associativity and commutativity of $\star$. Therefore
$\Wick{A}$ satisfies \eqref{defWick3}. The converse follows 
from the fact that \eqref{defWick3} defines by recurrence a unique family.
\end{proof}


\subsection{Back to simple subsets}
\label{ssect:simple}

As in Subsection \ref{ssect:subsets}, we can restrict the whole discussion to Wick polynomials associated to finite sets $B\in\MS(\CA)\cap\{0,1\}^\CA$ and their linear span $J$. Indeed, if $A\in J$ then $\Wick{A}=W(A)$ also belongs to $J$ and is defined by the recursion
\[
	A  = \sum_{B_1\cdot B_2=A}  \E(X^{B_1})  \Wick{{B_2}}.
\]
As in Theorem \ref{wick:repr}, we have $W=\mu^{-1}\star\id$ and $\id=\mu\star W$, and, as 
in Proposition \ref{rec},
\[
\begin{split}
	& \Wick{A} = A + 
	\\ & +  \sum_{n\geq 1} (-1)^n \sum_{B\in\MS(\CA)} 
	\sum_{B_1,\ldots,B_{n}\in\MS(\CA)\setminus\{\emptyset\}} 
	\ind{(B\cdot B_1\cdots B_n=A)} \, \mu(B_1)\cdots\mu(B_n)\,  B
\end{split}
\]
for all $A\in \MS(\CA)\cap\{0,1\}^\CA$.

However, as we have seen in Section \ref{sect:deform} above, it is more interesting to work on 
the bialgebra $H$ than on the coalgebra $J$, see in particular Theorem \ref{ww}.


\section{On the inverse of unital functionals}
\label{sect:inverse}

As we have seen in Theorem \ref{wick:repr}, the element $\mu^{-1}\in{\mathcal G}(H)$ plays
an important role in the Hopf algebraic representation \eqref{W} of Wick products. From \eqref{inverse1} we obtain a general way to compute $\mu^{-1}$. 
Let us consider now a linear functional $\lambda:H\to\R$ which is also a unital algebra morphism (or {\it character}), namely such that $\lambda(\emptyset)=1$ and $\lambda(A\cdot B)=\lambda(A)\lambda(B)$ for all $A,B\in H$. Then there is a simple way to compute its inverse: namely as $\lambda^{-1}=\lambda\circ S$, where $S:H\to H$ is the {\it antipode}, i.e. the only linear map such that 
\[
	S\star \id=\id\star S=\emptyset \, \varepsilon,
\]
where $\emptyset$ is the unit and $\varepsilon$ the counit of $H$.
However the functional $\mu$ we are interested in 
is not a character (moments are notoriously
not multiplicative in general) and this simple representation is not available. 

The aim of this section is to provide an alternative
antipode formula, allowing to write $\mu^{-1}$ as $\hat\mu\circ \hat S$, where $\hat S:\hat H\to \hat H$ is the antipode of 
another Hopf algebra $\hat H$, $\hat\mu:\hat H\to{\mathbb R}$ is a suitable linear functional
and $H$ is endowed with a comodule structure over $\hat H$, see \eqref{invant} for  
the precise formulation. 
This way of representing $\mu^{-1}$ by means of a comodule structure
is directly inspired by \cite{bhz,Hai-Cha}, see Section \ref{sect:final}. 


\begin{definition}\label{iota} 
Let $\hat H$ be the free commutative unital algebra (the algebra of polynomials) generated by $\MS(\CA)$. We denote by $\bullet$ the product in $\hat H$ and we define the coproduct $\hat \Delta:\hat H\to \hat H\otimes \hat H$ given by $\hat\Delta (\iota A) = (\iota\otimes\iota)\Delta A$ and 
\[
	\hat\Delta(A_1\bullet A_2\bullet \cdots\bullet A_n) 
	= (\hat\Delta A_1)\bullet (\hat\Delta A_2)\bullet \cdots \bullet(\hat\Delta A_n),
\]
where $\iota: H \to \hat H$ is the canonical injection (which we will omit whenever this does not cause confusion). The unit of $\hat H$ is $\emptyset$ and the counit is defined by $\hat\varepsilon(A_1\bullet A_2\bullet \cdots\bullet A_n)=\varepsilon(A_1)\cdots\varepsilon(A_n)$.
\end{definition}

Since $\hat H$ is a polynomial algebra, $\hat\Delta$ is well-defined by specifying its action on the elements of $\MS(\CA)$, and requiring it to be multiplicative. It turns the space $\hat H$ into a connected graded Hopf algebra, where the grading is
\[
	|A_1\bullet A_2\bullet \cdots\bullet A_n|:=|A_1|+|A_2|+\cdots+|A_n|.
\]

The antipode $\hat S:\hat H\to \hat H$ of $\hat H$ can be computed by recurrence with the classical formula
\[
	\hat S A = - A - \sum_{B_1,B_2 \in\MS(\CA)\setminus \{\emptyset\}} 
	\binom{A}{B_1\, B_2} \left[\hat S B_1\right] \bullet B_2,
\]
where we dropped the injection $\iota$ for notational convenience. A closed formula for $\hat S$ follows
\begin{equation}\label{antip}
	\hat S A = -A+\sum_{n\geq 2}   
	(-1)^n \sum_{B_1,\ldots,B_n\not=\emptyset} \binom{A}{B_1\dots B_n} 
	\, B_1\bullet B_2 \bullet \cdots \bullet B_n.
\end{equation}
We denote by ${\mathcal C}(\hat H)$ the set of characters on $\hat H$. This is a group for the $\hat\star$ convolution, dual to $\hat\Delta$.

\begin{proposition}\label{restriction}
The restriction map $R:{\mathcal C}(\hat H)\to{\mathcal G}(H)$, $R\hat\lambda:=\hat\lambda|_H$ defines a group isomorphism.
\end{proposition}

\begin{proof}
The map is clearly bijective, since a character on $\hat H$ is uniquely determined by its values on $H$, and every $\lambda\in {\mathcal G}(H)$ gives rise in this way to a $\hat\lambda\in{\mathcal C}(\hat H)$ such that $R\hat\lambda=\lambda$.

It remains to show that $R$ is a group morphism. This follows from 
\[
	R(\hat\alpha\ \hat\star\ \hat\beta)(A)=(\hat\alpha\otimes\hat\beta)\hat\Delta A= 
	(\hat\alpha|_H\otimes\hat\beta|_H)\Delta A= (R\hat\alpha)\star(R\hat\beta)(A),
\]
where $\hat\alpha,\hat\beta\in{\mathcal C}(\hat H)$ and $A\in H$.
\end{proof}

For all $\lambda\in{\mathcal G}(H)$ we write $\hat\lambda$ for the only character on $\hat H$ which is mapped to $\lambda$ by the isomorphism $R$. By the previous proposition we obtain, in particular, that $(\hat\lambda)^{-1}|_H=\lambda^{-1}$ for all $\lambda\in{\mathcal G}(H)$. Since $\hat\lambda$ is a character on $\hat H$, we have 
$(\hat\lambda)^{-1}=\hat\lambda\circ \hat S$. Therefore 
\begin{equation}\label{invant}
	\lambda^{-1}=(\,\hat\lambda\circ \hat S\,)|_H.
\end{equation}
This formula can be used specifically to compute the inverse $\mu^{-1}$ of the functional $\mu$ in \eqref{W}.


\subsection{A comodule structure}
\label{ssect:comodule}

The above considerations suggest that we can introduce the following additional structure: if we define $\delta:H\to\hat H\otimes H$, $\delta:=(\iota\otimes\id)\Delta$, where $\iota:H\to\hat H$ is the canonical injection of Definition \ref{iota}, then $H$ is turned
into a {\it left comodule} over $\hat H$, namely we have
\begin{equation}
\label{comodule}
	(\hat\Delta\otimes\id_H)\,\delta=(\id_{\hat H}\otimes\delta)\,\delta, \qquad
	\id_H=(\hat\varepsilon\otimes\id_H)\,\delta,
\end{equation}
see \eqref{deltarho} above. Note that \eqref{comodule} is in fact just the coassociativity and counitality of $\Delta$ on $H$ in disguise. 

Then we can rewrite the Hopf algebraic representation \eqref{W} of Wick polynomials as follows:
\begin{equation}\label{W2}
	\Wick{A}=(\hat\mu\circ\hat S\otimes\id)\,\delta A, 
\end{equation}
for $A\in H$, where $\hat\mu$ is the $\bullet$-multiplicative extension of $\mu$ from $H$ to $\hat H$.
Expanding this formula by means of 
the closed formula for $\hat S$, one recovers, by different means, Proposition  \ref{rec}.

\medskip
From Proposition \ref{com1} above, we obtain

\begin{proposition}\label{com2}
We define the action of ${\mathcal C}(\hat H)$ on $H$ by
\[
	\psi_{\hat\lambda}:H\to H, \qquad \psi_{\hat\lambda}(A)
	=(\hat\lambda\otimes\id)\,\delta A, \qquad \hat\lambda\in {\mathcal C}(\hat H),
\] 
for $A\in H$. Then $\psi_{\hat\lambda}$ is comodule morphism for all $\hat\lambda\in{\mathcal C}(\hat H)$, namely
\[
\delta\circ\psi_{\hat\lambda}=(\id_{\hat H}\otimes\psi_{\hat\lambda})\, \delta.
\]
\end{proposition}


\section{Deformation of pointwise multiplication}
\label{sect:pointwise}

We show now that the ideas of the previous sections can be generalized and used to
define deformations of other products. The main example for us is the pointwise product on
functions $f:\R^d\to\R$, and we explain in the next sections how these ideas appear in
regularity structures.

Let us consider a fixed family $T=(\tau_i, \, i\in I)$. We denote by $\T$ the free commutative monoid on $T$, with commutative product $\odot$ and neutral element $\emptyset \in \T \setminus T$. We define also $(C,\odot,\emptyset)$ as the unital free commutative algebra generated by $T$; then $C$ is the vector space freely generated by $\T$.

 Elements of $\T$ are commutative monomials in the elements of $T$, i.e. a generic element in $\T$ is of the form
  \[ \tau = \tau_{i_1}\dotsm\tau_{i_k} \]
  where $\tau_{i_1},\dotsc,\tau_{i_k}\in T$ and juxtaposition denotes their commutative product in $\T$. The empty set $\emptyset$ plays the role of the unit.
  Elements of the free commutative algebra $C$ are simply linear combinations of these monomials.
  In this context the $\odot$ product is the bilinear extension to $C$ of the product in the monoid $\T$.

We denote now by $\CC:=\CC(\R^d)$ the space of continuous functions $f:\R^d\to\R$, for a fixed $d\geq 1$. We endow $\CC$ with the associative commutative product $\cdot$ given by the pointwise multiplication and consider the spaces  
\[
\CC^T:=\{\Pi:T\to\CC\}, \qquad \CC^\T:=\{\Gamma:\T\to\CC\},
\] 
of functions from $T$, respectively $\T$, to $\CC$. Any function $\Gamma\in\CC^\T$ can be uniquely extended to a linear map $\Gamma:C\to\CC$. One can think of $\CC^T$ as a space of $T$-indexed functions: this is typically what happens in perturbative expansions indexed by combinatorial objects (sequences, in usual Taylor expansions or in Lyons' classical theory of geometric rough paths, or more complex objects such as trees or forests, as in Gubinelli's theory of non-geometric rough paths or in Hairer's theory of regularity structures, for example).

When deforming perturbative expansions parametrized by combinatorial objects, it is useful and often necessary to keep track of the combinatorial indices since the deformation of the various terms of the expansion will depend in practice on both the terms themselves and on their indices. 
To implement this idea, we will, for $\Pi\in \CC^T$ and $\Gamma \in \CC^\T$, use the physicists' notation
\begin{align*}
	\langle\Pi,\tau\rangle&\in\CC, \quad \tau\in T,  \\
	\langle\Gamma,\alpha\rangle&\in\CC, \quad \alpha\in \T.
\end{align*}
for the evaluation of $\Pi$ and $\Gamma$ on $\tau$, respectively $\alpha$ (so that $\langle\Pi,\tau\rangle$ and $\langle\Gamma,\alpha\rangle$ are simply elements of $\CC$).
We use instead, for a given $\Gamma\in\CC^\T$, the classical functional notation 
\[
\Gamma(\alpha):=(\langle\Gamma,\alpha\rangle ,\alpha)\in \CC\times\T, \qquad \alpha\in \T,
\]
to denote a {\it copy} of the function $\langle\Gamma,\alpha\rangle$ indexed by $\alpha$.
\begin{definition}\label{C_G}
For every $\Gamma\in\CC^\T$ we denote by $C_\Gamma$ the vector space freely 
generated by $(\Gamma(\alpha), \alpha\in \T)$. 
\end{definition}
By definition, $C_\Gamma$ is isomorphic to $C$ as a vector space (through the projection map $\Gamma(\alpha)\mapsto \alpha$).
We also have an {\it evaluation map} $\ev:C_\Gamma\to \CC$
\begin{equation}\label{ev}
C_\Gamma\ni \sum_{i=1}^n c_i\, \Gamma(\alpha_i) \mapsto
\ev\left(\sum_{i=1}^n c_i\, \Gamma(\alpha_i)\right):= \sum_{i=1}^n c_i\, \langle\Gamma,\alpha_i\rangle\in\CC.
\end{equation}


Concretely, the aim of Definition \ref{C_G} is to use deformations of the algebraic 
structure of $C$, in particular its product, and the isomorphism between $C$ and $C_\Gamma$, to transfer the deformations of $C$ to $C_\Gamma$ and ultimately to the associated functions in $\CC$. 

We insist on the fact that, in general, the way such a function (or products thereof) will be deformed will depend on its index.
A key point to keep in mind is indeed that the vector space $C_\Gamma$ is {\it not} isomorphic, in general, to the linear span of 
$(\langle\Gamma,\alpha\rangle, \alpha\in \T)$ in $\CC$. To take a trivial example,
we might choose $\langle\Gamma,\alpha\rangle=0\in\CC$ for some (or all) $\alpha\in \T$, while $\Gamma(\alpha)$ is 
always a non-zero element of $C_\Gamma$ for all $\alpha\in \T$. In practice, $C_\Gamma$ is isomorphic to the linear 
span of $(\langle\Gamma,\alpha\rangle, \alpha\in \T)$ in $\CC$ if and only if the family $(\langle\Gamma,\alpha\rangle: \alpha\in \T)$ is linearly independent in $\CC$, that is if and only if the evaluation map $\ev$ is injective.

\begin{definition}\label{lemGamma} We define the commutative and associative product $\M_{\Gamma}$ on $C_\Gamma$
as the only linear map $\M_{\Gamma}:C_\Gamma\otimes C_\Gamma\to C_\Gamma$ such that
\[
	\M_{\Gamma}(\Gamma(\alpha)\otimes\Gamma(\beta)):=\Gamma({\alpha\odot\beta}), \qquad \forall\ \alpha,\beta\in \T.
\]
\end{definition}

Then $(C_\Gamma,\M_{\Gamma},\langle\Gamma,\emptyset\rangle)$ is a commutative unital algebra. In other terms, we have extended the canonical isomorphism from $C$ to $C_\Gamma$ to an isomorphism of algebras
$(C,\odot)\to (C_\Gamma,\M_{\Gamma})$.

\begin{definition}\label{R}
 If $\Gamma\in\CC^\T$ is a unital algebra  morphism from $(\T,\odot)$ to $(\CC,\cdot)$, namely if
\begin{equation}\label{Gamma}
	\langle \Gamma,\emptyset\rangle=1, \qquad  
	\langle \Gamma,\tau_1\odot\tau_2\rangle
	=\langle\Gamma,\tau_1\rangle\cdot \langle\Gamma,\tau_2\rangle,  
\end{equation}
for all $\tau_1,\tau_2\in\T$, where $\cdot$ is the pointwise multiplication on $\CC$, then $\Gamma$ is called a \emph{character}. 
\end{definition}
\begin{definition}
Since $C$ coincides with the free algebra generated by $T$, for each map $\Pi$ in $\CC^T$ there is a unique character $R\Pi\in\CC^\T$, called the \emph{canonical lift of $\Pi$}, with
\[
	\langle\Pi,\tau\rangle=\langle R\Pi,\tau\rangle, \qquad \forall \, \tau\in T.
\]
 In particular, for every $\Pi\in \CC^T$, $\M_{R\Pi}$ is mapped by the evaluation map \eqref{ev} to the canonical pointwise product on $\CC$.
\end{definition}

\subsection{Constructing deformations}
We want now to define {\it deformations} of the products of Definition \ref{lemGamma}, taking inspiration from Section 
\ref{sect:deform}. 
 We suppose that $C$ is a left-comodule over a Hopf algebra $(\hat C,\bullet,\emptyset,\hat\Delta,\hat\varepsilon)$, with coaction $\delta:C\to \hat C\otimes C$ satisfying the analogue of \eqref{comodule}. We stress that the coaction $\delta$ is not supposed to be multiplicative with respect to the $\odot$ product in $C$.

We say that $\lambda:\hat C\to\R$ is {\it unital} if it is a linear functional such that $\lambda(\emptyset)=1$. Then we define, as in the previous section,
\begin{equation}\label{psi}
	\psi_\lambda:C\to C, \qquad \psi_\lambda:=(\lambda\otimes\id)\delta,  
\end{equation}
for every unital $\lambda:\hat C\to\R$. It is easy to see that
\[
	\psi_\lambda \circ \psi_{\lambda'} = \psi_{\lambda' \ \hat\star\  \lambda},
\]
where $\hat\star$ is the convolution product with respect to the coproduct $\hat\Delta:{\hat C} \to{\hat C}\otimes{\hat C}$. We then define the product $\odot_\lambda$ on $C$ as
\begin{equation}\label{odotl}
	\alpha \odot_\lambda \beta 
	:=\psi_\lambda^{-1}[(\psi_\lambda\alpha)\odot(\psi_\lambda\beta)], \qquad \alpha,\beta\in C,
\end{equation}
where $\psi_\lambda^{-1}=\psi_{\lambda^{-1}}$ and $\lambda^{-1}:\hat C\to\R$ is the inverse of $\lambda$ with respect to the $\hat\star$ convolution product. It is easy to see that the product $\odot_\lambda$ is associative and commutative, arguing as in the proof of  Theorem \ref{deformation}. The product $\odot_\lambda$ is in general different from $\odot$.

\begin{definition}\label{gW} By analogy with the Hopf algebraic interpretation of Wick products,
the map $\psi_\lambda^{-1}=\psi_{\lambda^{-1}}$ is called the generalized Wick $\lambda$-product map.
\end{definition}


We can now define deformations of the product $\M_\Gamma$ on $C_\Gamma$.
\begin{definition}\label{anoterdef}
For every $\Gamma\in\CC^\T$ and every unital $\lambda:\hat C\to\R$ we can define a product $\M^\lambda_{\Gamma}$ on $C_\Gamma$ by
\begin{equation}\label{MG}
	\M^\lambda_{\Gamma}(\Gamma(\alpha)\otimes\Gamma(\beta))
	:=\Gamma({\alpha\odot_\lambda\beta}), \qquad \forall\ \alpha,\beta\in C,
\end{equation}
such that $\Gamma:(C,\odot_\lambda)\to(C_\Gamma,\M^\lambda_{\Gamma})$ is an algebra isomorphism. 
\end{definition}

We say that $\M^\lambda_{\Gamma}$ is a $\lambda$-deformation of $\M_{\Gamma}$; if $\lambda$ is the counit $\hat\varepsilon$ of $\hat C$, the counitality property \eqref{comodule} of $\delta$ implies that $\psi_{\hat\varepsilon}$ is the identity map on $C$, hence $\M^{\hat\varepsilon}_{\Gamma}$ coincides with $\M_{\Gamma}$. In particular we have

\begin{definition}
For every $\Pi\in\CC^T$ we can define a $\lambda$-{deformation} $\cdot_\lambda:=\M^\lambda_{R\Pi}$ of the canonical product $\M_{R\Pi}$ on $C_{R\Pi}$, such that
\begin{equation}\label{Pideform}
	\Pi({\tau_1})\cdot_\lambda \cdots \cdot_\lambda\Pi({\tau_n}) 
	:= R\Pi({\tau_1\odot_\lambda\cdots\odot_\lambda\tau_n}),
\end{equation}
where $\tau_1,\ldots,\tau_n\in T$.
\end{definition}
We stress again that, unlike the pointwise multiplication $\cdot$, the $\lambda$-defor\-mation $\cdot_\lambda=\M^\lambda_{R\Pi}$ is not defined on $\CC$ but rather, for every fixed $\Pi\in\CC^T$, on $C_{R\Pi}$. As stated below Definition \ref{anoterdef}, the deformation $\M^{\hat\varepsilon}_{R\Pi}$ coincides with the canonical product $\M_{R\Pi}$ on $C_{R\Pi}$.

\begin{lemma}\label{isomo}
For every unital $\lambda:\hat C\to\R$ and $\Gamma\in\CC^\T$ 
the map 
\begin{equation}\label{Gal}
	\Gamma^\lambda:C\to C_\Gamma, \qquad
	\Gamma^\lambda:=\Gamma\circ \psi_\lambda^{-1}
\end{equation}
defines an algebra isomorphism from $(C,\odot)$ to $(C_{\Gamma},\M_{\Gamma}^\lambda)$.
\end{lemma}

\begin{proof}
By \eqref{MG}
\begin{equation}\label{gagal}
	\M^\lambda_{\Gamma}(\Gamma^\lambda(\alpha)\otimes\Gamma^\lambda(\beta))
	=\Gamma^\lambda({\alpha\odot\beta}), \qquad \forall\ \alpha,\beta\in C,
\end{equation}
and the claim follows.
\end{proof}

In particular for $\Gamma=R\Pi$ we obtain by \eqref{Pideform}
\begin{equation}\label{Pideforma}
	\Gamma^\lambda({\tau_1})\cdot_\lambda \cdots \cdot_\lambda\Gamma^\lambda({\tau_n}) := 
	\Gamma^\lambda({\tau_1\odot\cdots\odot\tau_n}),
\end{equation}
which is reminiscent of \eqref{wickdef}.


\begin{example}\label{ex:81}{\rm
In the setting of the previous sections, we can consider $T=\CA$ with $\CC=\CC(\R^\CA)$, so that in this case, up to a canonical isomorphism,
$C=H$ and $\hat C=\hat H$. Then the most natural choice of  $\Pi\in{\rm Hom}(\CA,\CC(\R^\CA))$ is given by $<\Pi,a>:=t_a$, where $t_a:\R^\CA\to\R$ is the evaluation of the $a$-component, and $R\Pi:H\to \CC$ is 
\[
	\langle R\Pi ,x_{a_1}\ldots x_{a_n}\rangle:=t_{a_1}\cdots t_{a_n}, \qquad a_1,\ldots,a_n\in\CA,
\]
Then \eqref{Pideforma} is the analogue of \eqref{wickdef} in this context, while \eqref{Pideform} defines a deformation $\cdot_\lambda$ of the pointwise product of $t_{a_1},\ldots ,t_{a_n}$. This point of view will be generalized in the next section.
}
\end{example}

We conclude this section with a remark. The above construction allows to construct families of deformed products
on a vector space $C_\Gamma$.
In the general case, the outcome of the product between $\Pi(\tau_1)$ and $\Pi(\tau_2)$ is not a function in $\CC$ but an element of $C_\Gamma$, namely a 
formal linear combination of functions in $\CC$ indexed by elements of $\T$. One may prefer a genuine function in $\CC$ as an outcome. 

If the family $(\langle\Gamma,\alpha\rangle: \alpha\in \T)$ is linearly independent in $\CC$, then we can canonically embed 
$C_\Gamma$ in $\CC$ by means of the evaluation map \eqref{ev} and the product 
\[
\ev\circ\M^\lambda_\Gamma(\ev^{-1}\cdot,\ev^{-1}\cdot):\ev(C_\Gamma)\otimes\ev(C_\Gamma)\to \ev(C_\Gamma)
\]
is then a commutative and associative product on the linear subspace of $\CC$ spanned by $(\langle\Pi,\tau\rangle, \tau\in T)$.

In the general case, one can still compute $\ev\circ\M^\lambda_\Gamma$, which indeed belongs to $\CC$, and obtain a linear map 
\[
\ev\circ\M^\lambda_\Gamma: C_\Gamma\otimes C_\Gamma\to\CC.
\]
In this case, $\ev\circ\M^\lambda_\Gamma$ 
is a weaker notion than a genuine product. We refer to the discussion on "multiplication" in regularity structures at the
beginning of \cite[Section 4]{reg}.


\section{Wick products of trees}
\label{sect:Wicktrees}

We now discuss the main example we have in mind of the general construction in Section \ref{sect:pointwise}, namely rooted trees, that are a generalization of classical monomials as we show below. With the application to rough paths in mind \cite{Gubinelli2010693}, we denote by $\T$ the set of all non-planar non-empty rooted trees with edges (not nodes) decorated with letters from a finite alphabet $\{1,\ldots,d\}$. We stress that all trees in $\T$ have at least one node, the root.

The set $\T$ is a commutative monoid under the associative and commutative {\it tree product} $\odot$ given by the identifications of the roots, e.g. 
\begin{equation}\label{planted}
	\Forest{[[i_1[i_2]][i_3]]}\ \odot\tIV{i_4}{i_5}{i_6} \ =\  \Forest{[[i_1[i_2]][i_3][i_4[i_5][i_6]]]},
\end{equation}
see also \cite[Definition 4.7]{bhz}.

The rooted tree $\tRoot$ with a single node and no edge is the neutral element for this product.
The set of monomials in $d$ commuting variables $X_1,\ldots,X_d$ can be embedded in $\T$ as follows: every
primitive monomial $X_i$ is identified with $\tI{i}$, and the product of monomials with the tree product. In this way
every monomial is identified with a decorated corolla, for instance
\begin{equation}\label{Cor}
	X_{i}X_j X_k \ \longrightarrow\ \Forest{[[i][j][k]]}.
\end{equation}
See the discussion around Lemma \ref{isoCor} below for more on this identification.

We denote by $T\subset\T$ the set of all non-planar {\it planted} rooted trees. We recall that a 
rooted tree is planted if its root belongs to a single edge, called the trunk. For example, in the left-hand side of \eqref{planted}, the first tree is not planted, while the second is.

We also denote by $\F$ the set of non-planar rooted forests with edges (not nodes) decorated with letters from the finite alphabet $\{1,\ldots,d\}$, such that every non-empty connected component has at least one edge. On this space we define the product $\bullet$ given by the disjoint union, with neutral element the empty forest $\emptyset$.


We perform the identification
\begin{equation}\label{ident}
	\tRoot=\emptyset
\end{equation}
between the rooted tree $\tRoot\in\T$ and the empty forest $\emptyset\in\F$. Then we obtain canonical embeddings 
\begin{equation}\label{embed}
	T\hookrightarrow\T \hookrightarrow \F
\end{equation}
and moreover
\begin{itemize}
\item $(\T,\odot)$ is the free commutative monoid on $T$, 
\item $(\F,\bullet)$ is the free commutative monoid on $\T$.
\end{itemize}
In both cases the element $\tRoot=\emptyset$ is the neutral element.
We denote by 
\begin{itemize}
\item $V$ the vector space generated freely by $T$, 
\item $C$ the vector space generated freely by $\T$,
\item $\hat C$ the vector space generated freely by $\F$.
\end{itemize}
Then we have 
\begin{itemize}
\item $(C,\odot)$ is the free commutative unital algebra generated by $T$, 
\item $(\hat C,\bullet)$ is the free commutative unital algebra generated by $\T$,
\end{itemize}
and again in both cases the element $\tRoot=\emptyset$ is the neutral element. Finally, by \eqref{ident} and \eqref{embed}
we also have canonical embeddings 
\begin{equation}\label{embed2}
V\hookrightarrow C \hookrightarrow \hat C.
\end{equation}

On $\hat C$ we also define the coproduct $\hat\Delta$, given by the extraction-contraction operator of arbitrary subforests \cite{CEFM11}:
\begin{equation}\label{hatDelta}
\hat\Delta \tau = \sum_{\sigma\subseteq\tau} \sigma\otimes \tau/\sigma, \qquad \tau\in\F,
\end{equation}
where a subforest $\sigma\in\F$ of $\tau$ is determined by a (possibly empty) subset of the set 
of edges of $\tau$, and $\tau/\sigma$ is the tree obtained by contracting each connected 
component of $\sigma$ to a single node. We recall that by \eqref{ident} the empty forest and the tree reduced
to a single node are identified and called $\emptyset$.
For example, 
\begin{align*}
  \hat\Delta\Forest{[ [i_1[i_2]] [i_3] ]} &= \Forest{[ [i_1[i_2]] [i_3] ]}\otimes\emptyset + \emptyset\otimes \Forest{[ [i_1[i_2]] [i_3] ]} + \tI{i_1}\otimes\tV{i_2}{i_3}+\tI{i_2}\otimes\tV{i_1}{i_3}+
  \tI{i_3}\otimes\tII{i_1}{i_2}\\
  &\qquad+\tV{i_1}{i_3}\otimes\tI{i_2}+\tII{i_1}{i_2}\otimes\tI{i_3}+\tI{i_2}\bullet\tI{i_3}\otimes\tI{i_1}.
\end{align*} 
If $\hat\varepsilon:\hat C\to\R$ is the linear functional such that $\hat\varepsilon(\sigma)=\ind{(\sigma=\emptyset)}$ for $\sigma\in\F$, then $(\hat C,\bullet,\emptyset,\hat\Delta,\hat\varepsilon)$ is a Hopf algebra \cite{CEFM11}. 

Note that, unlike $(\hat H,\hat\Delta)$ in Section \ref{sect:inverse}, $(\hat C,\hat\Delta)$ is not co-commutative; moreover the canonical embedding $C \hookrightarrow \hat C$ in \eqref{embed2} is not an algebra morphism from $(C,\odot)$ to $(\hat C,\bullet)$. We could also endow $C$ with a coproduct $\Delta_C$ (the extraction-contraction operator of a subtree at 
the root, which plays an important role in \cite{bhz} and is isomorphic to the classical Butcher--Connes--Kreimer coproduct), 
but we do not need this for what comes next.

\bigskip
We now go back to the construction of Section \ref{sect:pointwise}. With the embedding $C \hookrightarrow \hat C$, the coaction 
\[
    \delta:C\to\hat C \otimes C, \qquad
    \delta(\tau):=\hat\Delta\tau,
\]
makes $C$ a left-comodule over $\hat C$ by an analogue of Proposition \ref{com2}. Then we can define $\psi_\lambda:C\to C$ as in \eqref{psi}, for $\lambda:\hat C\to\R$ unital, and a deformed product $\odot_\lambda$ on $C$ as in \eqref{odotl} which is in general truly different from $\odot$.

For $\Pi\in\CC^T$ and $\Gamma:=R\Pi\in\CC^\T$ as in Definition \ref{R}, the map $\Gamma^\lambda=(R\Pi)\circ \psi_{\lambda^{-1}}$ defines by Lemma \ref{isomo} an algebra isomorphism from $(C,\odot)$ to $(C_{\Gamma},\M_{\Gamma}^\lambda)$, so that in particular we have the analogue of \eqref{Pideforma}.

This idea is very important in regularity structures, where the pointwise product of explicit (random) distributions is ill-defined, while a suitable deformed product is well-defined as a (random) distribution. The above construction allows to recover a precise algebraic structure of such deformed pointwise products, in the same spirit as Theorem \ref{ww}. See Section \ref{sect:final} below for a discussion. 

\medskip
We show now how these ideas can be implemented concretely, that is how a character $\lambda$ can be constructed in practice in some interesting situation, generalizing the construction of Wick polynomials in the previous sections of this article.

\bigskip
Let us now consider a $\CC^T$-valued random variable $X$, such that
\begin{itemize}
\item $\langle X,\emptyset\rangle =1$,

\item $X$ is stationary, i.e., $\langle X,\tau\rangle(\cdot+x)$ has the same law as $\langle X,\tau\rangle$ for all $x\in\R^d$ and $\tau\in T$

\item $\langle X,\tau\rangle(0)$ has finite moments of any order for all $\tau\in T$.

\end{itemize}

Then we can define 
\begin{equation}\label{muep}
	\mu :C\to\R, \qquad \mu(\tau):=\E(\langle RX,\tau\rangle(0)).
\end{equation}
There is a unique extension of $\mu$ to a linear $\hat\mu:\hat C\to\R$ which is a character of $(\hat C,\bullet)$ and we denote by $\hat\mu^{-1}:\hat C\to\R$ its inverse with respect to the $\hat\star$ convolution product. 

\begin{theorem}\label{thm92}
  Let $X$ be as above. The map $\lambda=\hat\mu^{-1}$ is the unique character on $(\hat C,\bullet)$ such that
\[
	\E(\langle RX, \psi_{\lambda}\tau\rangle(0))=0, \qquad \forall\, \tau\in \T\setminus\{\emptyset\}.
\]
\end{theorem}

\begin{proof}
We note first that for every character $\lambda$ on $(\hat C,\bullet)$ we have
\[
	\E(\langle RX, \psi_{\lambda}\tau\rangle(x))
	=(\lambda\otimes\mu)\hat\Delta\tau
	=(\lambda\ \hat\star\ \mu)\, \tau, \qquad \forall \, \tau\in \T. 
\]
In particular 
\[
	\E(\langle RX, \psi_{\hat\mu^{-1}}\tau\rangle(x))
	=(\hat\mu^{-1}\otimes\hat\mu)\hat\Delta\tau
	=0, \qquad \forall\, \tau\in \T\setminus\{\emptyset\}.
\]
On the other hand, since $\lambda$ and $\hat\mu$ are characters on $(\hat C,\bullet)$, if for all 
$\tau\in \T$
\[
	(\lambda\ \hat\star\ \hat\mu)\,\tau=\ind{(\tau=\emptyset)},
\]
then the same formula holds by multiplicativity for all $\tau\in\F$ and we obtain that $\lambda= \hat\mu^{-1}$.
\end{proof}

\begin{remark}{\rm
By stationarity, the function $RX\circ \psi_{\hat\mu^{-1}}\in\CC^\T$ has the additional property
\[
	\E(\langle RX, \psi_{\hat\mu^{-1}}\tau\rangle(x))=0, \qquad \forall\, \tau\in \T\setminus\{\emptyset\}, \
	x\in\R^d.
\]
In other words, $RX\circ \psi_{\hat\mu^{-1}}:C\to \CC$ gives a {\it centered} deformed product. This important
example is the exact analogue in this context of the {\it BPHZ renormalization} in regularity structures, see \cite[Theorem 6.17]{bhz}.
}
\end{remark}

We now show that the construction on decorated rooted trees generalizes in a very precise sense the Wick products of Section \ref{sect:wick}. We use the identification between monomials in $d$ commuting variables $X_1,\ldots,X_d$ and corollas decorated with letters from $\{1,\ldots,d\}$ that we have explained in \eqref{Cor}. Choosing $\CA:=\{1,\ldots,d\}$ we obtain a canonical embedding of $H\hookrightarrow C$, where $H$ is defined in Definition \ref{def:H}; we call ${\rm Cor}$ the image of $H$ in $C$ by this embedding. Then a simple verification shows that 

\begin{lemma}\label{isoCor}
The embedding $H\hookrightarrow C$ is a Hopf algebra isomorphism between $(H,\cdot,\emptyset,\Delta,\varepsilon)$ and $({\rm Cor},\odot,\tRoot\,,\hat\Delta,\hat\varepsilon)$, where $\hat\Delta$ is defined in \eqref{hatDelta}.
\end{lemma}

We obtain that every deformation $\odot_\lambda$ for a unital $\lambda:\hat C\to\R$ defines a product on ${\rm Cor}$ which is isomorphic to the deformed product defined in \eqref{eq:deformation} by restricting $\lambda$ to a map from ${\rm Cor}$ to $\R$.

\section{Connection with regularity structures}
\label{sect:final}

It would go beyond the scope of this work to introduce and explain the algebraic and combinatorial aspects of seminal theory of regularity structures \cite{reg}; we want at least to explain how the concept of {\it renormalization}, which plays such a prominent role there, is intimately related to the {\it deformation} of the standard pointwise product described in the previous sections. These ideas can also be found in the theory of {\it rough paths} \cite{Lyons98,MR2314753,Gubinelli200486,Gubinelli2010693}, which has largely inspired the theory of regularity structures. 

We denote by ${\mathcal D}'(\R^d)$ the classical space of distributions or generalized functions on $\R^d$.
The recent papers \cite{bhz,Hai-Cha} introduce a Hopf algebra $\hat H$ together with a linear space $H$, which is moreover a left-comodule over $\hat H$ with coaction $\delta:H\to\hat H\otimes H$. This framework is then used to describe in a compact way a number of complicated algebraic operations, related to the concept of {\it renormalization}. The space $H$ in \cite{bhz} is an expanded version of the linear span of decorated rooted trees $V$ defined in Section \ref{sect:Wicktrees} above; more precisely it is the vector space freely generated by a more complicated set of decorated rooted trees, which is aimed at representing monomials of generalized Taylor expansions. The space $\hat H$ in \cite{bhz} is a Hopf algebra of decorated forests with a condition of {\it negative homogeneity}. 

In \cite{bhz,Hai-Cha}, the linear space $H$ codes random distributions, which depend on a regularisation parameter $\epsilon > 0$. As one removes the regularisation by letting $\epsilon \to 0$, these random distributions do not converge in general. More precisely, we have (random) linear functions $\Pi_\epsilon:H\to{\mathcal D}'(\R^d)$ which are well defined for all $\epsilon>0$, but for which there is in general no limit as $\epsilon\to 0$. In fact, we even have $\Pi_\epsilon:H\to\CC(\R^d)$, and $\Pi_\epsilon$ is constructed in a multiplicative way as in Lemma \ref{R} above. Indeed, although $H$ is not an algebra, it is endowed with a {\it partial product}, i.e., some but not all pairs of its elements are supposed to be multiplied. We try to make this idea more precise in the next 

\begin{definition}
A \emph{partial product} on $H$ is a pair $(\M,S)$ where $S\subseteq H\otimes H$ is a linear space and $\M:S\to H$ is a linear function.
\end{definition}

Therefore, if $\tau$ and $\sigma$ are elements of $H$, their product $\M(\tau\otimes\sigma)$ is well defined if and only if $\tau\otimes\sigma\in S$. For example, in regularity structures one has an element $\Xi\in H$ such that $\Pi_\epsilon\Xi = \xi_\epsilon:=\rho_\epsilon*\xi$, where $\xi$ is a white noise on $\R^d$ (a random distribution in ${\mathcal D}'(\R^d)$) and $(\rho_\epsilon)_{\epsilon>0}$ is a family of mollifiers. Although $(\xi_\epsilon)^2$ is well-defined as a pointwise product in $\CC(\R^d)$, as $\epsilon\to0$ there is no limit in ${\mathcal D}'(\R^d)$ and indeed, we do not expect to multiply $\xi$ by itself in ${\mathcal D}'(\R^d)$. We express this by imposing that $\Xi\otimes\Xi\notin S$.

The divergences that arise in this context are due to ill-defined products; this is already clear in the example of $\Xi\otimes\Xi$ and $(\xi_\epsilon)^2$. Another more subtle example is the following: we consider $\xi_\epsilon:=\rho_\epsilon*\xi$ again, and a (possibly random) function $f_\epsilon:\R^d\to R$ which, as $\epsilon \to 0$, tends to a non-smooth function $f$. Then the pointwise product $f_\epsilon\cdot \xi_\epsilon$ does not converge in general, since the product $f\cdot \xi$ is ill-defined in ${\mathcal D}'(\R^d)$. However a proper deformation of this pointwise product may still
be well defined in the limit.

Let $(\tau_i, i\in I)\subset H$ be a family that freely generates $H$ as a linear space. We can now give the following 

\begin{definition}\label{Pip}
Let $\Pi:(\tau_i, i\in I)\to{\mathcal D}'(\R^d)$ be a map and $(\M,S)$ a {partial product} on $H$. We define $C_\Pi$ as the vector space freely generated by  the symbols $(\Pi(\tau_i), i\in I)$ as in Definition \ref{C_G}, and $\Pi:H\to C_\Pi$ the unique linear extension of $\tau_i\mapsto\Pi(\tau_i)\in C_\Pi$. Then we define a {partial product} on $C_\Pi$ as follows:
\begin{itemize}
	\item $S_\Pi\subseteq C_\Pi\otimes C_\Pi:=\{\Pi(\tau)\otimes\Pi(\sigma): \tau\otimes\sigma\in S\}$\\[-0.2cm]

	\item $\M_\Pi:S_\Pi\to C_\Pi$, $\M_\Pi(\Pi(\tau)\otimes\Pi(\sigma)):=\Pi({\M(\tau\otimes\sigma)})$.
\end{itemize}
\end{definition}

We are clearly inspired by
the construction of the previous sections, by realizing that we can work on distributions rather than on continuous functions. We stress that this definition allows to define partial products of distributions in a very general setting. 

However the construction of interesting $\Pi:(\tau_i, i\in I)\to{\mathcal D}'(\R^d)$ may not be simple. The method which is successfully used in a large class of applications in \cite{reg,bhz,Hai-Cha} is the following. We start from a $\Pi_\epsilon:(\tau_i, i\in I)\to\CC(\R^d)$ which is multiplicative in the sense that
\[
	\langle\Pi_\epsilon, \M(\tau\otimes\sigma)\rangle
	=\langle\Pi_\epsilon,\tau\rangle\cdot\langle\Pi_\epsilon,\sigma\rangle, \qquad \forall \, \tau\otimes\sigma\in S,
\]
where $\cdot$ is the standard pointwise product in $\CC(\R^d)$. In order to obtain a convergent limit as $\epsilon\to 0$, we try to deform this pointwise product, using the comodule structure of $H$ over $\hat H$. For all unital multiplicative and linear $\lambda :\hat H\to\R$ we define $\psi_{\lambda}:H\to H$ as in \eqref{psi} and then we set as in \eqref{Gal}
\[
	\Pi^{\lambda}_\epsilon(\tau):= \Pi_\epsilon(\psi_{\lambda}^{-1}\tau), \qquad \tau\in H.
\]
Then we can define the deformed partial product on $C_{\Pi_\epsilon}$:
\[
	\M_{\Pi_\epsilon}^{\lambda}(\Pi^{\lambda}_\epsilon(\tau)\otimes \Pi^{\lambda}_\epsilon(\sigma)):=
	\Pi^{\lambda}_\epsilon(\M(\tau\otimes\sigma)), \qquad \tau\otimes\sigma\in S.
\]
If $\lambda=\lambda_\epsilon$ is chosen in such a way that $\Pi^{\lambda_\epsilon}_\epsilon$ converges to a well defined map $\hat\Pi:(\tau_i, i\in I)\to{\mathcal D}'(\R^d)$, then we can define on $C_{\hat\Pi}$ the partial product
\[
	\M_{\hat\Pi}(\hat\Pi(\tau)\otimes\hat\Pi(\sigma)):=
	\hat\Pi(\M(\tau\otimes\sigma)), \qquad \tau\otimes\sigma\in S
\]
which is the analogue of \eqref{Pideforma} in this setting. We note that in general neither $\Pi_\epsilon$ nor $\lambda_\epsilon$ converge; indeed, $\lambda_\epsilon$ diverges exactly in a way that compensates the divergence of $\Pi_\epsilon$, in such a way that $\Pi^{\lambda_\epsilon}_\epsilon$ converges.

The fact that the above construction can indeed be implemented in a large number of interesting situations is the result of \cite{bhz,Hai-Cha}. Those papers consider random maps $\Pi_\epsilon$ with suitable properties which resemble those of $X$ in Theorem \ref{thm92}, namely $\Pi_\epsilon$ is supposed to be stationary and to possess finite moments of all orders. Then, as in Theorem \ref{thm92}, it is possible to choose a specific element $\lambda_\epsilon:\hat H\to\R$ which yields a {\it centered} family of functions $\Pi_\epsilon\circ \psi_{\lambda_\epsilon^{-1}}$, see \cite[Theorem 6.17]{bhz}. Under very general conditions, this special choice produces a converging family as $\epsilon\to0$ \cite{Hai-Cha}. 

Therefore the renormalized (converging) random distributions are a {\it centered} version of the original (non-converging) ones. The specific functional $\lambda_\epsilon^{-1}$ is equal to $\mu_\epsilon\circ{\mathcal A}$, where $\mu_\epsilon:\hat H\to\R$ is an expectation with respect to $\Pi_\epsilon$ as in \eqref{muep}, and ${\mathcal A}$ is a {\it twisted antipode}; the functional $\mu_\epsilon\circ{\mathcal A}$ plays the role which is played by $\hat\mu^{-1}$ in Theorem \ref{thm92}.

\begin{remark}{\rm
We stress that the centered family $\Pi_\epsilon\circ \psi_{\lambda_\epsilon^{-1}}$ can not be in general reduced to the Wick polynomials of Theorem \ref{wick:repr}. This is because the coaction $\delta:H\to\hat H\otimes H$ in this context is significantly more complex than \eqref{coprodH} and \eqref{hatDelta}.
}
\end{remark}

\bibliography{mybib}
\bibliographystyle{plain}
\vspace{-1.5ex}
\end{document}